\documentclass[11pt]{amsproc}

\usepackage{amsmath}
\usepackage{amssymb}
\usepackage{amsthm}
\usepackage{amsfonts}

\RequirePackage{srcltx}
\def\leq {\leqslant}
\def\le {\leqslant}
\def\ge {\geqslant}
\def\geq {\geqslant}

\usepackage{amssymb}
\usepackage{amsfonts}
\usepackage{amsmath}
\usepackage{graphicx}%
\providecommand{\U}[1]{\protect\rule{.1in}{.1in}}
\theoremstyle{plain}

\newtheorem{lemma}{Lemma}
\newtheorem{theorem}{Theorem}
\newtheorem{corollary}{Corollary}
\newtheorem*{thma}{Theorem A}
\newtheorem*{thma'}{Theorem A$'$}
\newtheorem*{pr1}{Problem 1}
\newtheorem*{pr2}{Problem 2}

\theoremstyle{definition}
\newtheorem*{remark}{Remark}
\newtheorem*{remark1}{Remark}

\numberwithin{equation}{section}
\begin{document}
\title[Weighted norm inequalities for convolution and Riesz potential]
 {
 Weighted norm inequalities for convolution \\and Riesz potential
 }
\author{Erlan Nursultanov}
\address{E. Nursultanov \\
 Lomonosov Moscow State University (Kazakh Branch) and
Gumilyov Eurasian National University, Munatpasova 7
\\010010
Astana Kazakhstan} \email{er-nurs@yandex.ru}
\author{Sergey Tikhonov}
\address{S. Tikhonov \\
ICREA
and Centre de Recerca Matem\`{a}tica (CRM), E-08193, Bellaterra, Barcelona
} \email{stikhonov@crm.cat}

\thanks{
This research was partially supported by the MTM2011-27637/MTM, 2009 SGR 1303, RFFI 12-01-00169, NSH 979.2012.1, and
 Ministry of Education and Science of the Republic of Kazakhstan (0112RK02176, 0112RK00608).
}



\begin{abstract}
In this paper, we prove analogues of O'Neil's inequalities for the convolution in the weighted Lebesgue spaces.
We also establish the weighted two-sided norm inequalities for the potential operator.
\end{abstract} \maketitle

\vskip 0.5cm
\section{Introduction}


In this paper, we study the convolution  of functions on $\mathbb{R}^n$:
\begin{equation*}
(K*f)(x)=\int_{\mathbb{R}^n} K(x-y) f(y)dy, \qquad x\in \mathbb{R}^n.
\end{equation*}

The following generalization of Young's inequality is due to O'Neil  (\cite {ON}) : if
 $1 < p < q < \infty, \;\;\; 0<\tau_1\leq \tau_2\leq\infty$, and $\frac1{r} = 1 - \frac1p + \frac1q$, then
\begin{equation}\label{1}
L^{p,\tau_1}*L^{r,\infty}\subset L^{q,\tau_2},
\end{equation}
where
$L^{p,\tau}=L^{p,\tau}(\Bbb R^n)$ is the Lorentz space (see, e.g., \cite[Ch 4]{be}).
Throughout the paper the expression of form $X*Y\subset Z$ involving function spaces $X,Y,$ and $Z$ means that
whenever $f\in X$, $K\in Y$, then $f*K\in Z$ and
$$ \|f*K\|_Z\le C\|f\|_X \|K\|_Y,
$$
where the constant $C$ being independent of $f$ and $K$.


A generalization of Young's inequality for the weighted  $L^p$ spaces  was obtained in \cite{Kerman}.
 The weighted Lebesgue space
$L^{p}(\omega)=L^{p}(\omega;\Bbb R^n)$ consists of all measurable functions such that
$\|f\|_{
L^{p}(\omega)}
=(\int_{\mathbb{R}^n} |f|^p \omega^{p})^{{1}/{p}}<\infty$, where the weight $\omega$ is a nonnegative locally integrable function.
Moreover,
for the case of power weight $\omega(x)=|x|^s$, we write $L^{p}(\omega; \Bbb R^n)=L^{p}_s(\Bbb R^n)$.

\begin{thma} \cite{Kerman} We have
 \begin{eqnarray}\label{emb1}
 L^{p}_\alpha(\Bbb R^n)*L^{\theta}_{s}(\Bbb R^n)\subset L^{q}_{-\beta}(\Bbb R^n),\qquad 1<p,q,\theta<\infty,\qquad
  \frac{1}{q}\le \frac{1}{p}+\frac{1}{\theta},
\end{eqnarray}
 provided
\begin{eqnarray}\label{parameters1}
 \frac{1}{q}=\frac{1}{p}+\frac{1}{\theta}+\frac{\alpha+\beta+s}{n}-1,
\end{eqnarray}
\begin{eqnarray}\label{parameters2}\alpha<n/p',\quad  \beta<n/q,\quad s<n/\theta',
\end{eqnarray}
and
\begin{eqnarray}\label{parameters3}
\alpha+  \beta\ge 0,\quad \alpha+  s\ge 0, \quad  \beta+s\ge 0.
\end{eqnarray}
\end{thma}
Convolution inequalities $ L^{p}_\alpha(\Bbb R^n)*L^{\theta}_{s}(\Bbb R^n)\subset L^{q}_{-\beta}(\Bbb R^n)$ were also studied in \cite{bui}.
   Conditions (\ref{parameters1})--(\ref{parameters3}) are necessary as it is shown in Section \ref{final}.

Let us now formulate the first question studied in this paper:
\begin{pr1}
Find sufficient conditions on  weights $\mu$ and $\nu$, so that
$$L^p (\mu) *L^{r,\infty}\subset L^q(\nu), \qquad
1 < p \le q < \infty, \quad 1<r\le \infty.
$$
\end{pr1}
Further, we deal with an important example of the convolution operator $K*f$, where  $K(z)=|z|^{\gamma-n},$   $z\in \Bbb R^n$, i.e.,
  the operator of the fractional integration (or Riesz potential):
\begin{equation}\label{poten}
I_{\gamma}f(x) =
\int\limits_{\Bbb R^n} \frac{f(y)}{ {|x-y|^{n-\gamma}}}dy, \qquad  0<\gamma<n.
\end{equation}
For this operator, (\ref{1}) implies   $I_{\gamma}: L^{p,\tau_1}(\Bbb R^n) \to L^{q,\tau_2}(\Bbb R^n)$  for $\frac1p - \frac1q =\frac{\gamma}{n}$
and $0<\tau_1\le \tau_2\le\infty$.
In particular, this yields the Hardy-Littlewood-Sobolev theorem, i.e.,  $I_{\gamma}: L^{p}(\Bbb R^n) \to L^{q}(\Bbb R^n)$.
Continuity properties of the  potential operator in the Lebesgue spaces are well known, see, e.g.,
\cite{bs}, \cite{Sob}, \cite{Ste3}.

Analogous  questions have been intensively investigated in the weighted Lebesgue spaces.
For the power weights,  Stein and Weiss \cite{st} generalized the Hardy-Littlewood-Sobolev theorem as follows: 
{\it Suppose
\begin{equation}\label{stein-weiss1}
\gamma =\alpha+\beta + n(1/p-1/q), \;\;\alpha< n/p', \;\;\beta < n/q, \;\; 0\le \alpha+\beta,
\end{equation}
then}
 \begin{equation}\label{stein-weiss}
I_\gamma : L^{p}_\alpha(\Bbb R^n)\to L^{q}_{-\beta}(\Bbb R^n),
 \qquad
 1< p\le q<\infty.
\end{equation}
For arbitrary weights $\mu$ and $\nu$, the solution of the
 problem stated as ``give necessary and sufficient conditions on  weights
 for the Riesz potential $I_{\gamma}$
to be bounded from $L^p(\mu, \Bbb R^n)$ to $L^q(\nu, \Bbb R^n)$'', can be found in
\cite{kokil}, \cite{lac}, \cite{maz}, \cite{m-w}, \cite{per}, 
 \cite{sawyer}, \cite{saw}, \cite{saww} and references therein.
In particular, it is known (see \cite[(2.2)]{lac}) that for $1<p\le q<\infty,$
\begin{eqnarray}\label{99}
\qquad\;\;
\|I_{\gamma}\|_{L^{p}(\mu, \Bbb R^n)\to L^{q}(\nu, \Bbb R^n)}
\asymp
\|I_{\gamma}\|_{L^{q'}(\nu^{1-q'}, \Bbb R^n)\to L^{p',\infty}(\mu^{1-p'}, \Bbb R^n)}
+
\|I_{\gamma}\|_{L^{p}(\mu, \Bbb R^n)\to L^{q, \infty}(\nu, \Bbb R^n)}.
\end{eqnarray}
Note also that `upper' triangle case with one weight was considered in \cite{maz} and \cite{maz1} in terms of capacity;
a non-capacity characterization was proved in
\cite{casc}.

Our second goal in this paper is to study two-sided norm inequalities
for the weighted potential operator
\begin{eqnarray}\label{poten}
\qquad\qquad
(A_\gamma f)(y) = (A_{\gamma,\mu,\nu}f)(y) =
\nu(y)\int\limits_{\Bbb R^n} \frac{f(x)\mu(x)}{ {|x-y|^{n-\gamma}}}dx, \qquad y\in \Bbb{R}^n, \qquad 0<\gamma<n,
\end{eqnarray}
in the Lorentz spaces. More precisely, we study the following question. 
\begin{pr2}
Find upper and lower bounds  of the norm of the operator $A_{\gamma}: L^{p,\tau_1}(\Bbb R^n)\to L^{q, \tau_2}(\Bbb R^n)$,
$1 < p, q < \infty,$ $1<\tau_1, \tau_2\le \infty,$ in terms of $\mu$ and $\nu$.
\end{pr2}


Note that this question is similar to the following problem:
to estimate  the norm of the potential  operator in the weighted Lorentz space, i.e.,
$I_{\gamma}:$ ${L^{p, \tau_1}({\mu^{-1}}, \Bbb R^n)\to L^{q, \tau_2}({\nu}, \Bbb R^n)}$, which is, in general, a different question.
The latter case was thoroughly studied by Kerman for the power weights; see Theorems 4.1 and 4.5 in  \cite{Kerman}.



In this paper $F \lesssim G$ means that $F \le C G$; by $C$ we denote positive constants that may be different
on different occasions. Moreover, $F \asymp G$ means that $F \lesssim G \lesssim F$.

For a Lebesgue measurable set $E$, $|E|$ will denote its Lebesgue measure and $\nu(E) =
\int_E \nu(x) dx$ will denote its weighted measure. If $E$ and $W$ are any subsets of $\mathbb{R}^n$, the Minkowski sum of $E$ and $W$ is defined by $E+W=\Big\{x+y: \;x\in E\;\mbox{and}\;y\in W \Big\}$.
The characteristic function of a set $E$ is denoted by  $\chi_{E}$.
Let also $\overline{0}=(0,\ldots,0)$, $\overline{N}=(N,\ldots,N)$, and $B_r({\overline{N}})=\Big\{x\in\mathbb{R}^n: |x-{\overline{N}}|\le r\Big\}$.
As usual, $\psi^*$ is  the decreasing rearrangement of $\psi$ see, e.g., \cite{bs}.

Finally,  $X'$ is the associate space of the space $X$, i.e.,
 $X'=\Big\{g:\|g\|_{X'}=\sup\limits_{\|f\|_X\le 1}|\int fg|<\infty\Big\}$.

\vskip 0.8 cm

\section{Main results}
First, we give a solution of Problem 1 by proving the following extension of Theorem A.

\begin{theorem} \label{TTT} Let  $1 < p \le q < \infty$ and $1<r< \infty$.
Let weights $\mu$ and $\nu$ satisfy, for any $\lambda\ge 1$,
\begin{eqnarray}\label{1zv}
\mu^*(\lambda t)\lesssim \frac{\mu^*(t)}{\lambda^\alpha},\qquad t>0,
\end{eqnarray}
\begin{eqnarray}\label{2zv}
\nu^*(\lambda t)\lesssim \frac{\nu^*(t)}{\lambda^\beta},\qquad t>0,
\end{eqnarray}
for some $\alpha\ge 0$ and $\beta\ge 0$ such that
\begin{eqnarray}\label{25zv}
\alpha+ 1/r+1/p>1,   \qquad \beta+ 1/r+1/q'>1.
\end{eqnarray}
Then
a sufficient condition for
\begin{eqnarray}\label{zvezda}
L^p (\mu^{-1}) *L^{r,\infty}\subset L^q(\nu)
\end{eqnarray}
to hold is
\begin{eqnarray}\label{21}
\mathcal{G} :=\sup\limits_{|E|=|W|}\frac{\nu(E) \mu(W)}{|E|^{1+\frac1r+\frac1p - \frac1q }}< \infty,
\end{eqnarray}
where the supremum is taken over all measurable sets $E$ and $W$ of the same positive measure.
A necessary condition for (\ref{zvezda}) is
\begin{eqnarray}\label{21-1}
\mathcal{S} :=\sup\limits_{|E|=|W|}\frac{\nu(E) \mu(W)}{|E|^{1+\frac1p - \frac1q}|E-W|^{\frac1r}}< \infty,
\end{eqnarray}
where the supremum is taken over all measurable sets $E$ and $W$ of the same positive measure.
\end{theorem}

\begin{remark}\label{C3}
 Note that conditions (\ref{1zv}) and (\ref{2zv}) are fulfilled for arbitrary weights $\mu$ and $\nu$ with $\alpha=\beta=0$.
Hence, the statement of Theorem \ref{TTT} holds for arbitrary $\mu$ and $\nu$ if
$1/r+1/p>1$ and $1/r+1/q'>1$, cf. (\ref{25zv}). Similar remarks can be applied to Theorem \ref{T} and Corollary \ref{C2} below.
Let us also note that if a function $\psi$
is  $\delta$-regularly varying  (see \cite[Sect. 2]{bir}), that is,
it satisfies $\psi(\lambda t)\lesssim \frac {\psi(t)}{\lambda^{\delta}}$, then it also satisfies
$\psi^*(\lambda t)\lesssim \frac {\psi^*(t)}{\lambda^{\delta}}$.
\end{remark}




For
 the power weights $\mu(x)=|x|^{-\alpha}$ and $\nu(x)=|x|^{-\beta}$, Theorem \ref{TTT} implies
necessary and sufficient conditions for
(\ref{zvezda}) to hold.
\begin{corollary}\label{cor-kerman}
Suppose  $1 < p \le q < \infty$ and $1<r< \infty$.
Then
\begin{eqnarray}\label{emb2}
L^{p}_\alpha(\Bbb R^n)*L^{r,\infty}(\Bbb R^n)\subset L^{q}_{-\beta}(\Bbb R^n)
\end{eqnarray}
if and only if
\begin{eqnarray}\label{emb2-cond1}
\frac{1}{q}=\frac{1}{p}+\frac{1}{r}+\frac{\alpha+\beta}{n}-1,
\quad 0\le \alpha<n/p',\quad  0\le \beta<n/q.
\end{eqnarray}
In particular, this implies the statement of Theorem A in case when $\alpha,\beta,s\ge 0$.
\end{corollary}

If $\alpha=\beta=0$, (\ref{emb2}) is O'Neil's inequality
$L^{p}(\Bbb R^n)*L^{r,\infty}(\Bbb R^n)\subset L^{q}(\Bbb R^n)$, $\frac{1}{q}=\frac{1}{p}+\frac{1}{r}-1$.
 Moreover, (\ref{emb2}) gives the Stein-Weiss inequality
 (\ref{stein-weiss}) under conditions (\ref{stein-weiss1}) and
   $\alpha,\beta\ge 0$.

\vskip 0.6 cm

Let us now study Problem 2.
 The following result gives upper and lower norm estimates  of
the weighted potential operator $A_\gamma$ in the Lorentz spaces.

\begin{theorem} \label{T} Let  $1 < p, q < \infty$, $0 < \tau_1\leq\tau_2 \le \infty$, and $0<\gamma<n$.
\\
{\textnormal{(A)}}.\; Let
the weights $\mu$ and $\nu$ satisfy conditions (\ref{1zv}) and (\ref{2zv})
for some $\alpha\ge 0$ and $\beta\ge 0$ such that
\begin{eqnarray}\label{26zv}
\alpha+ 1/p>\gamma/ n,   \qquad \beta+ 1/q'>\gamma/ n.
\end{eqnarray}
 If
\begin{eqnarray}\label{1.2}
\mathcal{L}:=\sup\limits_{|E|=|W|}\frac{\nu(E) \mu(W)}{|E|^{2+\frac1p - \frac1q -\frac{\gamma}{n}}}< \infty,
\end{eqnarray}
where the supremum is taken over all measurable sets $E$ and $W$ of the same positive measure,
 then $A_\gamma$ is bounded from $L^{p,\tau_1}(\Bbb R^n)$ to $L^{q,\tau_2}(\Bbb R^n)$
and
$$
\|A_\gamma\|_{L^{p,\tau_1}(\Bbb R^n) \to L^{q,\tau_2}(\Bbb R^n)} \le C \mathcal{L},
$$
where $C$ depends only on parameters  $p, q$,  $\tau_1$,   $\tau_2$, $\alpha$, $\beta$, $\gamma$, and $n$.
\\
{\textnormal{(B)}}.\;
If the operator $A_\gamma$ is bounded from $L^{p,\tau_1}(\Bbb R^n)$ to $L^{q,\tau_2}(\Bbb R^n)$, then
\begin{eqnarray}\label{1.3}
\qquad
\mathcal{F}:=\sup\limits_{B\in M} \frac{\nu(B) \mu(B)}{|B|^{2+\frac{1}{ p}-\frac{1}{
q}-\frac \gamma n }}\le \,C\, \|A_\gamma\|_{L^{p,\tau_1}(\Bbb R^n) \to L^{q,\tau_2}(\Bbb R^n)},
\end{eqnarray}
where
$M$ is  the collection of all balls in $\Bbb R^n$ and
$C$ depends only on parameters  $p, q$,  $\tau_1$,   $\tau_2$, $\alpha$, $\beta$, $\gamma$, and $n$.
\end{theorem}
\begin{remark}\label{C3}
Comparing $\mathcal{S}$ and $\mathcal{F}$ defined by (\ref{21-1}) and (\ref{1.3}), we see that
$$
\mathcal{F}
\le
\sup\limits_{x\in \Bbb R^n, {{E}}\,\,\small{\mbox{convex}}}\frac{\nu(E) \mu(E+x)}{|E|^{1+\frac1r+\frac1p - \frac1q }}
\lesssim \mathcal{S}.
$$
\end{remark}
Taking $\tau_1=p$ and $\tau_2=q$ in Theorem \ref{T}, we get
$ \mathcal{F} \lesssim \|A_\gamma\|_{L^{p}(\Bbb R^n) \to L^{q}(\Bbb R^n)} \lesssim
\,\mathcal{L}.$
Moreover, under certain regularity conditions on  weights, the left-hand side bound $\mathcal{F}$ and the right-hand side bound $\mathcal{L}$
 are equivalent. In particular, this holds
 for radial weights $\psi(|x|)$ satisfying the following condition:
\begin{equation} \label{AA}
\psi^\ast(t) \lesssim \frac{1}{t} \int_{t/2}^{t} \psi(s)\,ds, \qquad t>0.
\end{equation}
Note that (\ref{AA}) holds for any monotonic function  $\psi$ on $(0,\infty)$.
It is easy to give an example  of a non-monotonic function satisfying condition (\ref{AA}), for example,
$$
\psi(t)=\frac {|\cos n t|}{t^\alpha}, \qquad t>0, \quad 0<\alpha<1.
$$
Theorem \ref{T} implies the following relation for the norm of the weighted potential operator.

\begin{corollary}\label{C2}  Let $p,  q, \gamma, \tau_1, \tau_2$ satisfy all conditions of Theorem \ref{T}.
Suppose functions $\nu_0(\cdot)$ and $\mu_0(\cdot)$ defined on $(0,\infty)$ satisfy condition (\ref{AA}). Then, for the operator
\begin{eqnarray*}
(A_\gamma f)(y) = (A_{\gamma,\mu,\nu}f)(y) =
\nu_0(|y|)\int\limits_{\Bbb R^n} \frac{f(x)\mu_0(|x|)}{ {|x-y|^{n-\gamma}}}dx,
\qquad y\in \Bbb{R}^n, \qquad 0<\gamma<n,
\end{eqnarray*}
one has
\begin{eqnarray}\label{26zv0-0}
\|A_\gamma \|_{L^{p,\tau_1}(\Bbb R^n)\to L^{q,\tau_2}(\Bbb R^n)} \asymp
\sup\limits_{a>0} \frac1{a^{n(2+\frac{1}{ p}-\frac{1}{
q}-\frac \gamma n )}} \int\limits_{0}^{a} \nu_0(r)r^{n-1}dr\int\limits_{0}^{a}\mu_0(r)r^{n-1}
dr.
\end{eqnarray}
In particular, if $1<p\le q<\infty$, then
\begin{eqnarray*}
\|A_\gamma \|_{L^{p}(\Bbb R^n)\to L^{q}(\Bbb R^n)} \asymp
\sup\limits_{a>0} \frac1{a^{n(2+\frac{1}{ p}-\frac{1}{
q}-\frac \gamma n )}} \int\limits_{0}^{a} \nu_0(r)r^{n-1}dr\int\limits_{0}^{a}\mu_0(r)r^{n-1}
dr.
\end{eqnarray*}
\end{corollary}

In the case of the power weights $\mu(x)=|x|^{-\alpha}$ and $\nu(x)=|x|^{-\beta}$ we obtain necessary and sufficient conditions for boundedness of the operator $A_\gamma$.

\begin{corollary}\label{C111} Let $1 < p , q < \infty$, $\alpha\geq 0$, and $\beta\geq 0$.
\\
{\textnormal{ (A)}.} Let either $0< \tau_1\leq\tau_2<\infty$ or $1<\tau_1\leq\tau_2\leq\infty$. Then the operator
$$
(A_{\gamma, \alpha, \beta} f)(y) =
\int\limits_{\Bbb R^n} \frac{f(x)} {|x|^\alpha|x-y|^{n-\gamma}|y|^\beta}dx,
\qquad y\in \Bbb{R}^n, \qquad 0<\gamma<n,
$$
is bounded from $L^{p,\tau_1}(\Bbb R^n)$ to $L^{q,\tau_2}(\Bbb R^n)$,  if and only if
\begin{eqnarray}\label{26zv0}
\gamma=\alpha+\beta+n\left(\frac1p-\frac1q\right), \qquad \alpha< \frac n{p'},\qquad \beta <\frac nq.
\end{eqnarray}
{\textnormal{ (B)}.} Let
$0< \tau\le 1.$
The operator $A_{\gamma, \alpha, \beta}$
is bounded from $L^{p,\tau}(\Bbb R^n)$ to $L^{q,\infty}(\Bbb R^n)$ if and only if
\begin{eqnarray}\label{26zv1}
\gamma=\alpha+\beta+n\left(\frac1p-\frac1q\right), \qquad \alpha\le \frac n{p'},\qquad \beta \le\frac nq.
\end{eqnarray}
\end{corollary}
In particular, for the Lebesgue spaces this result is reduced to 
$I_\gamma : L^{p}_\alpha(\Bbb R^n)\to L^{q}_{-\beta}(\Bbb R^n),$
 $1< p\le q<\infty,$ that is,
(\ref{stein-weiss}).
 Note that necessity of conditions (\ref{stein-weiss1}) was discussed in \cite{duo}.

\vskip 0.8 cm

\section{Lemmas}
Define \begin{eqnarray*}
f^{**}(t)=\frac 1t \int\limits_0^{t}f^*(s) ds.
\end{eqnarray*}
We will use the fact (see \cite[p. 53]{bs}) that  $f^{\ast\ast}$  can be written as follows:
\begin{eqnarray*}
f^{\ast\ast}(t)=\sup_{{|E|=t}}\frac1{|E|}
\int\limits_E |f(x)|dx.
  \end{eqnarray*}

\begin{lemma}\label{L1}  Let  $W$ and $E$ be measurable sets in $\Bbb R^n$, then
\begin{eqnarray}\label{2}
\qquad \left|\int_{W}\mu(x)\int_{E}\nu(y)K(y-x)dy dx\right|
\leq
\int_{0}^{|W|}\mu^*(s)\int_0^{|E|} \nu^*(t)K^{**}\left(\max(t,s) \right)dt ds.
\end{eqnarray}
\end{lemma}
\begin{proof}
The proof is inspired by ideas from the paper \cite{jfaa}.
By the Hardy-Littlewood rearrangement inequality
\begin{eqnarray*}
\int_{W}\mu(x)\int_{E}\nu(y)K(y-x)dy dx
&\leq&
\int_{0}^{|W|} \mu^*(s)\left(\int_{E}\nu(y)K(y-\cdot)dy \right)^*(s) ds
\\
&\leq&
\int_{0}^{|W|}\mu^*(s)\left(\int_{E}\nu(y)K(y-\cdot)dy \right)^{**}(s)ds
\\
&=& \int_{0}^{|W|}\mu^*(s)\sup_{|\eta_1|=s}\frac1{|\eta_1|}\int_{\eta_1}\left|\int_{E}\nu(y)K(y-x)dy\right|dx ds
\\
&\leq&
\int_{0}^{|W|}\mu^*(s)\sup_{|\eta_1|=s}\Big(\frac1{|\eta_1|}\int_{E}|\nu(y)|\int_{\eta_1}|K(y-x)|dxdy\Big)
ds.
\end{eqnarray*}
We use similar estimates for the inner integral to get
\begin{eqnarray*}
\int_{W}\mu(x)\int_{E}\nu(y)K(y-x)dy dx
&\leq&
\int_0^{|W|}\mu^*(s)\sup_{|\eta_1|=s}\int_0^{|E|}\nu^*(t)\sup_{|\eta_2|=t}\frac1{|\eta_2|}\frac1{|\eta_1|}
\int_{\eta_2}\int_{\eta_1}\left|K(y-x)\right|dxdy dt ds\\
&\leq&
\int_0^{|W|}\mu^*(s)\int_0^{|E|}\nu^*(t)\sup_{|\eta_1|=s}\sup_{|\eta_2|=t}
\frac1{|\eta_2|}\frac1{|\eta_1|}\int_{\eta_2}\int_{\eta_1}\left|K(y-x)\right|dxdy dtds.
\end{eqnarray*}
If $t\geq s$, then

\begin{align*}
&\sup_{|\eta_1|=s}\sup_{|\eta_2|=t}\frac1{|\eta_2|}\frac1{|\eta_1|}\int_{\eta_2}\int_{\eta_1}|K(y-x)|dxdy\\
&\leq
\sup_{|\eta_1|=s}\frac1{|\eta_1|}\int_{\eta_1}\sup_{|\eta_2|=t}\frac1{|\eta_2|}\int_{\eta_2}|K(y-x)|dydx =K^{**}(t).
\end{align*}
Similarly, if  $t<s$, then
$$
\sup_{|\eta_1|=s}\sup_{|\eta_2|=t}\frac1{|\eta_2|}\frac1{|\eta_1|}\int_{\eta_2}\int_{\eta_1}|K(y-x)|dxdy
\leq K^{**}(s)
$$
and the statement follows.
\end{proof}

\begin{lemma}\label{L2}
Let  $\beta>1$. Suppose
\begin{eqnarray}\label{int-zv}
\mathcal{B}=\sup\limits_{\omega>0}\frac{1}{\omega^{\beta}}\int_{0}^{\omega}\mu^*(s)ds\int_{0}^\omega \nu^*(t)dt \int_{0}^\omega K^*(s)ds<\infty.
\end{eqnarray}
Then there exists  $C$ depending only on  $\beta$ such that
\begin{eqnarray}\label{23}
\sup_{\omega>0}\frac1{\omega^{\beta-1}}\int_{0}^{\omega}\mu^*(s)\int_0^\omega \nu^*(t)K^{**}\left(\max(t,s) \right)dt ds\leq C \mathcal{B}.
\end{eqnarray}
\end{lemma}

\begin{proof}
For  $\omega>0$ we have
\begin{align*}
&\frac1{\omega^{\beta-1}}\int_{0}^{\omega}
\mu^*(s)\int_{0}^{\omega} \nu^*(t)K^{**}\left(\max(t,s) \right) dtds
\\&=\frac1{\omega^{\beta-1}}\int_{0}^{\omega}
\mu^*(s)
\left(\int_{0}^s \nu^*(t)K^{**}(s)dt+\int_s^\omega \nu^*(t)K^{**}(t) dt\right)ds
\\
&=\frac1{\omega^{\beta-1}}\int_0^\omega
K^{**}(s)\left(\int_0^s \nu^*(t)dt \int_{0}^s\mu^*(t)dt\right)'ds
\\
&=
\left.\frac1{\omega^{\beta-1}}\left[K^{**}(s)\int_{0}^s \nu^*(t)dt\int_{0}^s\mu^*(t)dt\right|_{0}^\omega-\right.
\left. \int_0^\omega\left(K^{**}(s)\right)'\int_{0}^s \nu^*(t)dt\int_0^s\mu^*(t)dtds\right]\\
&=:
I_1+I_2.
\end{align*}
To estimate  $I_2$, taking into account that
$$-\left(K^{**}(s)\right)'=\frac 1s\Big(K^{**}(s)-K^*(s)\Big)\le \frac 1s K^{**}(s),$$
we get
\begin{eqnarray*}
I_2&\leq&\frac1{\omega^{\beta-1}}\int_0^\omega K^{**}(s)\left(\int_0^s \nu^*(t)dt\int_0^s\mu^*(t)dt\right)\frac{ds}{s}.
\end{eqnarray*}
By (\ref{int-zv})
\begin{eqnarray*}
I_2&\leq& \frac{\mathcal{B}}{ \omega^{\beta-1}}\int_{0}^\omega\frac{ds}{s^{2-\beta}}\lesssim \mathcal{B}.
\end{eqnarray*}
Let us estimate  $I_1$. We have
\begin{eqnarray*}
I_1&=&\frac1{\omega^\beta}\int_{0}^\omega K^*(t)dt\int_{0}^{\omega}\nu^*(t)dt\int_{0}^{\omega}\mu^*(t)dt-\frac1{\omega^{\beta-1}}\lim\limits_{s\rightarrow0+}\frac1{s}\int_{0}^s K^*(t)dt\int_{0}^s \nu^*(t)dt\int_{0}^s\mu^*(t)dt\\&\le&
\mathcal{B}
+\frac1{\omega^{\beta-1}}\left|\lim\limits_{s\rightarrow0+}\frac1{s}\int_{0}^s K^*(t)dt\int_{0}^s \nu^*(t)dt\int_{0}^s\mu^*(t)dt\right|.
\end{eqnarray*}
Since
$$
\frac1{s}\int_{0}^s K^*(t)dt\int_0^s \nu^*(t)dt\int_0^s\mu^*(t)dt\leq s^{\beta-1}\mathcal{B},\qquad s>0,
$$
then
$$
\lim\limits_{s\rightarrow0+}\frac1{s}\int_{0}^s K^*(t)dt\int_0^s \nu^*(t)dt\int_0^s \mu^*(t)dt=0,
$$
and therefore, $I_1\leq \mathcal{B}.$
\end{proof}

\begin{lemma} \label{L3} Let  $1 < p , q < \infty$.
 Let $\mu, \nu,$ and  $K$ be locally integrable functions on  $\Bbb R^n$ and satisfy, for any $\lambda\ge 1$,
\begin{eqnarray}\label{3zv}
\qquad \mu^*(\lambda t)\lesssim \frac{\mu^*(t)}{\lambda^\alpha},
\qquad
\nu^*(\lambda t)\lesssim \frac{\nu^*(t)}{\lambda^\beta},
\qquad
K^*(\lambda t)  \lesssim \frac{K^*(t)}{\lambda^\rho},\qquad t>0,
\end{eqnarray}
for some $\alpha\ge 0,$ $\beta\ge 0$, and $\rho\ge 0$.
Let also
\begin{eqnarray}\label{4zv}
\rho+\alpha+1/p\geq1,   \qquad \rho +\beta+1/q'\geq1.
\end{eqnarray}
If
\begin{eqnarray*}
\mathcal{D}:=\sup\limits_{t>0}\frac{1}{t^{2+1/p-1/q}}\int_{0}^{t}\mu^*(s)ds\int_{0}^{t} \nu^*(s)ds\int_{0}^{t} K^*(s)ds<\infty,
\end{eqnarray*}
then
\begin{eqnarray*}
 \sup_{{}^{\xi>0}_{\eta>0}}\frac{1}{\eta^{\frac1{q'}}}\frac{1}{\xi^{\frac1p}}\int_{0}^{\eta}\nu^*(t)\int_{0}^{\xi}\mu^*(s)K^{**}(\max(s,t))dsdt \le C\mathcal{D},
\end{eqnarray*}
where $C$ depends only on  $p, q, \alpha, \beta, \rho$.
\end{lemma}
\begin{proof}

 Suppose also that $\eta>\xi>0.$ We take an integer
$N>1$ such that $(N-1)\xi<\eta\leq N\xi.$
Then
\begin{align*}
&\frac{1}{\eta^{\frac1{q'}}}\frac{1}{\xi^{\frac1p}}\int_{0}^{\eta}\nu^*(t)\int_{0}^{\xi}\mu^*(s)K^{**}(\max(s,t))dsdt
\\
&\leq\frac{1}{\left((N-1)\xi\right)^{\frac1{q'}}}\frac{1}{\xi^{\frac1p}}\int_{0}^{\eta}\nu^*(t)\int_{0}^{\xi}\mu^*(s)K^{**}(\max(s,t))ds dt
\\
&\lesssim
\frac{1}{
{((N-1)\,\xi)^{1/q'}\xi^{1/p}}} \int\limits_{0}^{\xi} \mu^*(s)
\sum_{k=1}^N \int\limits_{(k-1)\xi}^{k\xi}
\nu^*(t)K^{**}(\max(s,t)) dtds.
\end{align*}
We divide the last expression into two terms:
\begin{eqnarray*}
&& \frac{1}{ {((N-1)\,\xi)^{1/q'}\xi^{1/p}}} \int\limits_{0}^{\xi}
\mu^*(s) \int\limits_{0}^{\xi} \nu^*(t)K^{**}(\max(s,t))dtds
\\
&+&
 \frac{1}{ {((N-1)\,\xi)^{1/q'}\xi^{1/p}}} \int\limits_{0}^{\xi}
\mu^*(s) \sum_{k=2}^N \int\limits_{(k-1)\xi}^{k\xi}
\nu^*(t)K^{**}(t) dtds =:J_1+J_2.
\end{eqnarray*}
Note that (\ref{3zv}) implies
\begin{eqnarray}\label{33zv}
K^{**}(\lambda t)\lesssim \frac{K^{**}(t)}{\lambda^\rho}.
\end{eqnarray}
Indeed,
\begin{eqnarray*}
K^{**}(\lambda t)=\frac 1{\lambda t}\int_0^{\lambda t}K^*(s) ds&=&\frac 1{t}\int_0^{t}K^*(\lambda s) ds\\&\lesssim&
\frac 1{t}\int_0^{t}\frac{K^*(s)}{\lambda^\rho} ds=  \frac{K^{**}(t)}{\lambda^\rho}.
\end{eqnarray*}
Therefore, using $\nu^*(\lambda t)\lesssim {\lambda^{-\beta}}{\nu^*(t)}$ and (\ref{33zv}), we get
\begin{eqnarray}\label {11}
J_2
&\lesssim&
\frac{1}{ {((N-1)\,\xi)^{1/q'}\xi^{1/p}}} \int\limits_{0}^{\xi}
\mu^*(s) ds
\sum_{k=2}^N
\nu^*\big((k-1)\xi\big)K^{**}\big((k-1)\xi\big) \xi
\nonumber
\\
&\lesssim&
 \frac{1}{ {((N-1)\,\xi)^{1/q'}\xi^{1/p}}} \int\limits_{0}^{\xi}
\mu^*(s)\,ds
\int\limits_{0}^{\xi}{\nu^*(t)}dt K^{**}(\xi)\sum_{k=2}^N
\frac{1}{(k-1)^{\rho+\beta}}.
\end{eqnarray}
Noting that $\rho +\beta+1/q'\ge 1$, we get
$$
 J_2\lesssim \mathcal{D}.
$$
Estimating $J_1$, we use Lemma \ref{L2} to get $J_1\lesssim \mathcal{D}$.

Summing  the estimates for $J_1$ and $J_2$, we finally have
\begin{eqnarray*}
&& \frac{1}{\eta^{\frac1{q'}}}\frac{1}{\xi^{\frac1p}}\int_{0}^{\eta}\nu^*(t)\int_{0}^{\xi}\mu^*(s)K^{**}(\max(s,t))dsdt
\lesssim \mathcal{D}
\end{eqnarray*}
in the case when $\eta>\xi$.

If $\xi\ge \eta$, then we use similar estimates and the condition  $\rho+\alpha+1/p \ge 1$.

\hfill\end{proof}

\vskip 0.8 cm

\section{Proof of Theorem \ref{TTT}}

\begin{proof}
We consider the operator
\begin{eqnarray}\label{operator}
(A f)(y) = (A_{K, \mu,\nu}f)(y) =
\nu(y)\int\limits_{\Bbb R^n} {f(x)K(x-y)\mu(x)} \,dx, \qquad y\in \Bbb{R}^n.
\end{eqnarray}
Let us first prove that the operator $A$ is  a $(p, q)$ quasi-weak-type operator, i.e.,
\begin{eqnarray}\label{operator-a}
\|A\|_{L^{p,1}(\Bbb R^n)\to L^{q,\infty}(\Bbb R^n)} < C^*,
\end{eqnarray}
where $C^* = C(p,q, r) \mathcal{G}\|K\|_{L^{r,\infty}(\Bbb R^n)}$. 
By Corollary 4.1 from the paper \cite{NT}, we have
\begin{eqnarray}\label{operator-ab}
 \|A\|_{L^{p,1}(\Bbb R^n)\to L^{q,\infty}(\Bbb R^n)}\asymp
\sup_{{}^{|W|>0}_{|E|>0}}\frac1{|E|^{\frac1{q'}}}\frac1{|W|^{\frac1p}}
\Big|\int_{E}\nu(y)\int_{W}K(y-x)\mu(x)dxdy\Big|
\end{eqnarray}
and hence, it is enough to show that the latter is bounded.

Let $E$ and $W$ be any measurable sets from $\mathbb{R}^n$ of positive measure. Lemma \ref{L1} yields
\begin{align*}
&\frac1{|E|^{\frac1{q'}}}\frac1{|W|^{\frac1p}}\int_{E}\nu(y)\int_{W}K(y-x)\mu(x)dxdy\\
&\leq\qquad\frac{1}{|E|^{\frac1{q'}}}\frac{1}{|W|^{\frac1p}}\int_{0}^{|E|}\nu^*(t)\int_{0}^{|W|}\mu^*(s)K^{**}(\max(s,t))ds dt\\
&\leq
\qquad
\Big(\sup_{t>0}\frac1{t^{1-\frac1r}}\int_0^tK^*(s)ds
\Big)
\frac{1}{|E|^{\frac1{q'}}}\frac{1}{|W|^{\frac1p}}\int_{0}^{|E|}\nu^*(t)\int_{0}^{|W|}\mu^*(s)\frac1{\max(s,t)^{\frac1r}}ds dt
\\
&\lesssim \qquad \|K\|_{L^{r,\infty}} \,\frac{1}{|E|^{\frac1{q'}}}\frac{1}{|W|^{\frac1p}}\int_{0}^{|E|}\nu^*(t)\int_{0}^{|W|}\mu^*(s)\frac1{\max(s,t)^{\frac1r}}ds dt.
\end{align*}
Since the function $\varphi(x)=|x|^{-n/r}$ satisfies $\varphi^{**}(s)\asymp s^{-1/r}$, Lemma \ref{L3} gives
\begin{align*}
&\frac{1}{|E|^{\frac1{q'}}}\frac{1}{|W|^{\frac1p}}\int_{0}^{|E|}\nu^*(t)\int_{0}^{|W|}\mu^*(s)\frac1{\max(s,t)^{\frac1r}}ds dt
\\
&\lesssim \quad\;\sup\limits_{t>0}\frac{1}{t^{2+1/p-1/q}}\int_{0}^{t}\mu^*(s)ds\int_{0}^{t} \nu^*(s)ds\int_{0}^{t} \frac 1{s^{\frac1r}}ds
\,\lesssim\, \mathcal{G}
\end{align*}
 for any
sets $E$ and $W$ with positive measure,
provided
\begin{eqnarray}\label{22}
1 < p , q, r < \infty,\;\; \;\; \alpha+ 1/r+1/p>1,   \;\;\;\; \beta+1/r+1/q'>1.
\end{eqnarray}

We note that  the expression $\mathcal{G}$
 depends only on $1/p-1/q$.
 Further, since conditions  (\ref{22}) contain strict inequalities,
 one can choose two pairs $(p_0,q_0)$ and  $(p_1,q_1)$ such that
$1<p_0<p<p_1<\infty$, $1<q_0<q<q_1<\infty$, and
$$
1/p-1/q= 1/p_0-1/q_0=1/p_1-1/q_1, \;\;\;\; \alpha+ 1/r+1/p_i>1,   \;\;\;\; \beta+1/r+1/q'_i>1, \;\;\; i= 0, 1.
$$
Then, (\ref{operator-a}) holds for $(p_0,q_0)$ and $(p_1,q_1)$,
and therefore,
$$
A : L^{p_0,1} \to L^{q_0,\infty}
\quad\mbox{and} \quad
A: L^{p_1,1} \to
L^{q_1,\infty},
$$
where the norms are bounded by $\mathcal{G}\|K\|_{L^{r,\infty}}
$ up to a constant.
Then, by the interpolation theorem (see, e.g., \cite[Ch 5, \S 1]{bs}) we write
$$
A : L^{p(\theta),\tau_1} \to L^{q(\theta),\tau_1}
$$
and
$$
\|A\|_{L^{p(\theta),\tau_1}\to L^{q(\theta),\tau_1}} \leq C(\theta, p_i, q_i, r, \tau_1) \,\mathcal{G}\,
\|K\|_{L^{r,\infty}}
$$
for
$$
1/p(\theta)=(1-\theta)/p_0 + \theta/p_1,\;\;
1/q(\theta)=(1-\theta)/q_0 + \theta/q_1,\,\,\;\;0<\theta<1\;,\;\;\; 0<\tau_1\leq \infty.
$$
For some $0<\theta<1$ we have $p(\theta)=p$.  Since
$1/p_0-1/q_0=1/p_1-1/q_1=1/p(\theta)-1/q(\theta),$  in this
case  $q(\theta)$ coincides with $q.$

Finally,
$$
\|A\|_{L^{p,\tau_1}\to L^{q,\tau_2}} \lesssim
\|A\|_{L^{p,\tau_1}\to L^{q,\tau_1}} \lesssim
\,\mathcal{G} \|K\|_{L^{r,\infty}},
$$
when $\tau_1\le\tau_2$.
Taking $p=\tau_1$ and $q=\tau_2$ implies (\ref{zvezda}).

Let us show necessity of (\ref{21-1}). Suppose (\ref{zvezda}) holds. For measurable  sets $E$ and $W$ such that $|E|=|W|$, we put
$$K(x):=|E-W|^{-\frac1r}\chi_{E-W}(x).$$
Then $\|K\|_{L^{r,\infty}(\Bbb R^n)}=1$. Using   (\ref{operator-ab}), we get
\begin{eqnarray*}
1\gtrsim
 \|A\|_{L^{p}(\Bbb R^n)\to L^{q}(\Bbb R^n)}
\gtrsim
 \|A\|_{L^{p,1}(\Bbb R^n)\to L^{q,\infty}(\Bbb R^n)}
&\gtrsim& \frac1{|E|^{\frac1{q'}}}\frac1{|W|^{\frac1{p}}}\Big|\int_{E}\nu(y)\int_{W}K(y-x)\mu(x)dxdy\Big|
\\
&=&\frac1{|E|^{1+\frac1p-\frac1q}}\frac1{|E-W|^{\frac1{r}}}\int_{E}\nu(y)\int_{W}\mu(x)dxdy.
\end{eqnarray*}
Taking supremum over all $E$ and $W$ completes the proof.
\hfill\end{proof}

\begin{proof}[Proof of Corollary \ref{cor-kerman}]

First, we prove the sufficiency part. We
  take
$\mu(x)=|x|^{-\alpha}$ and $\nu(x)=|x|^{-\beta}$
and note that (\ref{25zv}) can be written as
$\alpha/n+ 1/r+1/p>1$ and   $\beta/n+ 1/r+1/q'>1.$ Further,
$\mathcal{G}<\infty$ only if $\frac{1}{q}=\frac{1}{p}+\frac{1}{r}+\frac{\alpha+\beta}{n}-1$ and
(\ref{emb2}) follows.

Second, to prove necessity of condition (\ref{emb2-cond1}) for a fixed $r\in(1,\infty)$,
we consider $K(x)=|x|^{\gamma-n}$ such that $\frac1r=1-\frac{\gamma}{n}$. Therefore,  $0<\gamma<n$ and
$K\in L^{r,\infty}(\Bbb R^n)$. Then the boundedness of the fractional integral  from $L^{p}_\alpha(\Bbb R^n)$
to
$L^{q}_{-\beta}(\Bbb R^n)$ gives (see \cite[Th. 5.1]{duo}) that
\begin{eqnarray}\label{emb2-cond}
\frac{1}{q}=\frac{1}{p}+\frac{1}{r}+\frac{\alpha+\beta}{n}-1,
\quad \alpha<n/p',\quad  \beta<n/q.
\end{eqnarray}
These conditions can be also verified using (\ref{21-1}).

We have to show that
$\alpha, \beta\ge 0$.
Let an integer $N$ be sufficiently large.
 Putting
 $K(x):=\chi_{B_2({\overline{N}})}(x)$, we have
$K\in L^{r, \infty}(\Bbb R^n)$, $r>1$.
If
$$
\|f*K\|_{L^{q}_{-\beta}(\Bbb R^n)}\lesssim\|f\|_{L^{p}_\alpha(\Bbb R^n)} \|K\|_{L^{r,\infty}(\Bbb R^n)},$$ 
then the operator
 $$
  A_{K, \alpha, \beta} f(y)=\frac{1}{|y|^{\beta}}\int\limits_{\mathbb R^n}K(x-y)f(x)\,
\frac{dx}{|x|^{\alpha}}
$$
is bounded from $L^{p}(\mathbb R^n)$ to $ L^q(\mathbb R^n)$.
Taking into account (\ref{operator-ab}),

\begin{eqnarray*}
1\gtrsim
 \|A_{K, \alpha, \beta} \|_{L^{p}(\mathbb R^n)\rightarrow L^q(\mathbb R^n)}
& \gtrsim&
  \|
A_{K, \alpha, \beta} \|_{L^{p, 1}(\mathbb R^n)\rightarrow L^{q, \infty}(\mathbb R^n)}
\\
&\asymp& \sup\limits_{\begin{array}{l}
 |E|>0\\
 |W|>0   ,
\end{array}}\frac{1}{|E|^\frac1{q'}|W|^\frac1{p}}\Big|\int\limits_E
\frac{1}{|y|^{\beta}}\int\limits_{W}
K(x-y)\, \frac{dx}{|x|^{\alpha}}
dy\Big|.
\end{eqnarray*}
Then, since $\beta < n/q<n$
\begin{eqnarray*}
1&\gtrsim&
\frac{1}{|B_1({\overline{0}})|^{\frac1{q'}}|B_2({\overline{N}})|^{\frac1p}}
\int_{B_1({\overline{0}})}\frac{1}{|y|^{\beta}}\int\limits_{B_2(\overline{N})}
K(x-y)\, \frac{dx}{|x|^{\alpha}}
dy
\\
&\geq&
C_{p,q}
\int_{B_1({\overline{0}})} |y|^{-\beta}dy
\int_{B_1({\overline{N}})}|x|^{-\alpha}dx
\gtrsim
N^{-\alpha}. 
\end{eqnarray*}
This gives $\alpha\ge 0$. Similarly, we get $\beta\ge 0$.

Finally, since $$L^{\theta}_s(\Bbb R^n)\subset L^{r,\infty}(\Bbb R^n)\quad \mbox{for} \quad\frac{1}{r}=\frac{s}{n}+\frac{1}{\theta}, \quad 0\le s<n/\theta',$$
(\ref{emb2}) gives (\ref{emb1}) provided (\ref{parameters1}), (\ref{parameters2}), and $\alpha,\beta,s\ge 0$.
\end{proof}

\vskip 0.8 cm

\section{Boundedness of the weighted convolution operators in Lorentz spaces}
In this section we deal with Problem 2 and study the weighted convolution operator
$A_K=A_{K, \mu,\nu}$  given by
 (\ref{operator}).
 In particular,
 if $K(x)=|x|^{\gamma-n},$ then $A_K=A_\gamma$, where $A_\gamma$ is defined by (\ref{poten}).
The next result gives sufficient conditions for the operator $A_K$ to be bounded from
$L^{p,\tau_1}(\Bbb R^n)$ to $L^{q,\tau_2}(\Bbb R^n)$ in terms of $\mu, \nu,$ and $K$.
 This implies  sufficient conditions for the boundedness of the operator
$A_\gamma$  given in Theorem \ref{T} (A).

\begin{theorem} \label{TT} Let  $1 < p , q < \infty$, \;\; $0 < \tau_1\leq\tau_2 \le \infty$.
 Suppose either,
 \\
 {\textnormal{(A)}} for any $\lambda\ge 1$,
\begin{eqnarray}\label{1zvd}
\qquad\mu^*(\lambda t)\lesssim \frac{\mu^*(t)}{\lambda^\alpha},
\qquad
\nu^*(\lambda t)\lesssim \frac{\nu^*(t)}{\lambda^\beta},
\qquad
K^*(\lambda t)\lesssim  \frac{K^*(t)}{\lambda^\rho},\qquad t>0,
\end{eqnarray}
where  $\alpha\ge 0,$ $\beta\ge 0, \rho\ge 0$, and
$$
\rho+\alpha+1/p>1,   \qquad \rho +\beta+1/q'>1,
$$
or, \\
 {\textnormal{(B)}} for any $\lambda\ge 1$,
\begin{eqnarray}\label{4zvd}\qquad
\mu^*(t) \lesssim {\mu^*(\lambda t)}{\lambda^{\bar\alpha}},
\qquad
\nu^*(t)\lesssim {\nu^*(\lambda t)}{\lambda^{\bar\beta}},
\qquad
K^*(t) \lesssim {K^*(\lambda t)}{\lambda^{\bar\rho}},\qquad t>0,
\end{eqnarray}
where  $\bar\alpha\geq0$, $\bar\beta\geq0$, $\bar\rho \geq 0$, and
$$
\bar\rho+\bar\alpha+1/p<1,   \qquad \bar\rho +\bar\beta+1/q'<1.
$$
If
\begin{eqnarray*}
\mathcal{D}
:=
\sup\limits_{t>0}\frac{1}{t^{2+1/p-1/q}}\int_{0}^{t}\mu^*(s)ds\int_{0}^{t} \nu^*(s)ds\int_{0}^{t} K^*(s)ds<\infty,
\end{eqnarray*}
then the operator $A_K$ is bounded from $L^{p,\tau_1}(\Bbb R^n)$ to $L^{q,\tau_2}(\Bbb R^n)$
and
$$
\|A_K\|_{L^{p,\tau_1}(\Bbb R^n) \to L^{q,\tau_2}(\Bbb R^n)} \le C \mathcal{D},
$$
where $C$ depends only on  $p, q, \tau_1, \tau_2$ and the corresponding parameters from
(\ref{1zvd}) or
(\ref{4zvd}).
\end{theorem}
\begin{proof}
As it was shown in the proof of Theorem \ref{TTT}, it is sufficient to obtain the following estimate:
$$
 \|A_K\|_{L^{p,1}(\Bbb R^n)\to L^{q,\infty}(\Bbb R^n)}\lesssim \mathcal{D}.
$$
(A) Under conditions  (\ref{1zvd}),  following the proof of Theorem \ref{TTT}, we have
$$
 \|A_K\|_{L^{p,1}(\Bbb R^n)\to L^{q,\infty}(\Bbb R^n)}
    \lesssim
   \sup_{|E|>0,|W|>0}\frac{1}{|E|^{\frac1{q'}}}\frac{1}{|W|^{\frac1p}}\int_{0}^{|E|}\nu^*(t)\int_{0}^{|W|}\mu^*(s)K^{**}(\max(s,t))ds dt
  \lesssim \mathcal{D},
$$
where the last estimate directly follows from Lemma 3.
\\
(B) Suppose conditions  (\ref{4zvd}) hold.
Let us verify that
\begin{eqnarray}\label{ert}   \sup_{|E|>0,|W|>0}\frac{1}{|E|^{\frac1{q'}}}\frac{1}{|W|^{\frac1p}}\int_{0}^{|E|}\nu^*(t)\int_{0}^{|W|}\mu^*(s)K^{**}(\max(s,t))ds dt
  \lesssim \mathcal{D}.
\end{eqnarray}
 Let  $0<|W|<|E|$ and ${|E|}/{|W|}=:\lambda>1$. Note that
 $$
 K^{**}\left(\max(\frac s\lambda,t)\right)\lesssim \lambda^{\bar \rho}K^{**}\left(\max(s,t)\right) \quad \mbox{for any}\quad s,t>0.
 $$
Indeed,  using monotonicity of $K^{**}$ and condition (\ref{4zvd}), we get
\begin{eqnarray*}
 K^{**}\left(\max\big( s/\lambda,t\big)\right)&=&\frac{1}{\max(s/\lambda,t)}\int_0^{\max(s/\lambda,t)}K^*(u)du
\\
&\leq&
\frac{\lambda}{\max( s,t)}\int_0^{{\max(s,t)}/{\lambda}}K^*(u)du
 \\& =&\frac{1}{\max(s,t)}\int_0^{\max(s,t)}K^*\big(\frac u\lambda\big)du
 \leq C\lambda^{\bar \rho}K^{**}\left(\max(s,t)\right).
\end{eqnarray*}
In view of  $\bar\alpha+\bar\rho<\frac1 {p'}$, we have
\begin{align*}
 &\frac{1}{|E|^{\frac1{q'}}}\frac{1}{|W|^{\frac1p}}\int_{0}^{|E|}\nu^*(t)\int_{0}^{|W|}\mu^*(s)K^{**}\left(\max(s,t)\right)ds dt
 \\
  &= \frac{\lambda^{\frac1 p-1}}{|E|^{\frac1{q'}+\frac1p}}\int_{0}^{|E|}\nu^*(t)\int_{0}^{|E|}\mu^*(\frac s\lambda)K^{**}\left(\max(\frac s\lambda,t)\right)ds dt
\\
  &\lesssim \frac{\lambda^{\bar\alpha+\bar\rho+\frac1 p-1}}{|E|^{\frac1{q'}+\frac1p}}\int_{0}^{|E|}\nu^*(t)\int_{0}^{|E|}\mu^*(s)K^{**}\left(\max(s,t)\right)ds dt
 \\
  & \leq\frac 1{|E|^{\frac1{q'}+\frac1p}}\int_{0}^{|E|}\nu^*(t)\int_{0}^{|E|}\mu^*(s)K^{**}\left(\max(s,t)\right)ds dt.
\end{align*}
Further, Lemma 2 implies
$$
\frac{1}{|E|^{\frac1{q'}}}\frac{1}{|W|^{\frac1p}}\int_{0}^{|E|}\nu^*(t)\int_{0}^{|W|}\mu^*(s)K^{**}\left(\max(s,t)\right)ds dt
 \lesssim \mathcal{D}.
$$
The case  ${|W|}>{|E|}$ is similar.
Thus,  we have shown (\ref{ert}), which concludes the proof.
\hfill
\end{proof}

\vskip 0.8 cm

\section{Proof of Theorem \ref{T} and its corollaries}
\begin{proof}[Proof of Theorem \ref{T}]
We start with (B).  If $A_\gamma$ is bounded from $L^{p, \tau_1}(\Bbb R^n)$ to $L^{q, \tau_2}(\Bbb R^n)$,
then it is $(p,q)$ weak-type operator, i.e.,
    $A_\gamma: L^{p,\,1}(\Bbb R^n)\rightarrow L^{q,\infty}(\Bbb R^n)$,
  and therefore, by (\ref{operator-ab}), we have
\begin{align*}
 \|A_\gamma\|_{L^{p,\tau_1}(\Bbb R^n)\to L^{q,\tau_2}(\Bbb R^n)}
&\gtrsim
\,
\|A_\gamma\|_{L^{p,\min (1,\tau_1)}(\Bbb R^n)\to L^{q,\infty}(\Bbb R^n)}
 \asymp\|A_\gamma\|_{L^{p, 1}(\Bbb R^n)\to L^{q,\infty}(\Bbb R^n)}
\\&\asymp \sup_{{}_{}^{|E|>0, |W|>0}} \frac{1}{ {|E|^{1/q'}}}
\frac{1}{ {|W|^{1/p}}} \int\limits_E \nu(y) \int\limits_W
\mu(x) |x-y|^{\gamma-n} dxdy \\
&\geq
 \sup_{B\in M} \frac1{|B|^{\frac{1}{ p}+\frac{1}{
q'}}}
\int\limits_B \nu(y)
 \int\limits_B\mu(x) |x-y|^{\gamma-n} dxdy\\
 &\gtrsim
 \sup\limits_{B\in M} \frac1{|B|^{2+\frac{1}{ p}-\frac{1}{
q}-\frac \gamma n }} \int\limits_B \nu(x)dx\int\limits_B \mu(x)
dx =\mathcal{F},
\end{align*}
that is, (\ref{1.3}) is true.

\vskip 0.3cm
The part (A) follows from Theorem \ref{TT} (A) for
$K(x)=|x|^{\gamma-n}$.
Taking 
 $\rho=1-\gamma/n$ implies $
  \|A_\gamma\|_{L^{p,\tau_1}(\Bbb R^n)\to L^{q,\tau_2}(\Bbb R^n)}
\lesssim
 \mathcal{D}\asymp \mathcal{L}$.
\end{proof}

\begin{proof}[Proof of Corollary \ref{C2}]
To show that $\mathcal{L}\lesssim \mathcal{F}$ it is enough to verify that, for a fixed positive $a$,
\begin{eqnarray}\label{df}
\sup_{|E|=a^n}\int_E\nu_0(|x|)dx \lesssim \int_{B_{a_1}(\overline{0})}\nu_0(|x|) \,dx
\end{eqnarray}
and
\begin{eqnarray}\label{dff}
\sup_{|W|=a^n}\int_W\mu_0(|x|)dx \lesssim \int_{B_{a_1}(\overline{0})}\mu_0(|x|) \,dx,
\end{eqnarray}
where $a_1/a=C(n)$.

Let us show the first inequality. We have
\begin{eqnarray*}
\sup_{|E|=a^n}\int_E\nu_0(|x|)dx  = \int_0^{a^n} \big( \nu_0(|\cdot|)\big)^*(t) \, dt  &\lesssim& \int_0^{a_1^n}\nu_0^*(t^{1/n})dt
\lesssim \int_0^{a_1}\nu_0^*(t) t^{n-1}dt.
\end{eqnarray*}
Using condition (\ref{AA}), we get
\begin{eqnarray*}
\sup_{|E|=a^n}\int_E\nu_0(|x|)dx  &\lesssim&  \int_0^{a_1}  \left(\frac 1t\int_{t/2}^t \nu_0(s) ds\right) t^{n-1}dt
\\
&\lesssim&  \int_0^{a_1}   \nu_0(t)  t^{n-1}dt \lesssim
\int_{B_{a_1}(\overline{0})}\nu_0(|x|) \,dx,
\end{eqnarray*}
i.e.,  (\ref{df}) follows.
Using (\ref{df}) and (\ref{dff}) we get
\begin{eqnarray*}
\mathcal{L}&=&\sup_{a>0}\sup_{|E|=|W|=a}\frac1{a^{2+\frac{1}{ p}-\frac{1}{q}-\frac \gamma n }}\int_W\mu_0(|x|)dx\int_E\nu_0(|x|)dx
\\
&\lesssim&
\sup_{r>0}\frac1{|
B_{r}(\overline{0})
|^{2+\frac{1}{ p}-\frac{1}{q}-\frac \gamma n }}\int_{B_{r}(\overline{0})}\mu_0(|x|)dx\int_{B_{r}(\overline{0})}\nu_0(|x|)dx\le \mathcal{F}.
\end{eqnarray*}
Finally, in view of $\mathcal{F}\le \mathcal{L}\lesssim \mathcal{F}$, the result follows from Theorem \ref{T}.
\end{proof}


\begin{proof}[Proof of Corollary \ref{C111}] (A)
Suppose  (\ref{26zv0}) holds.
Then
$\|A_{\gamma,\alpha, \beta} \|_{L^{p,\tau_1}(\Bbb R^n)\to L^{q,\tau_2}(\Bbb R^n)}<\infty$
follows from Corollary \ref{C2}.
Indeed, for
$\mu(x)=|x|^{-\alpha}$ and $\nu(x)=|x|^{-\beta}$, we have $A_{\gamma,\alpha, \beta}=
A_{\gamma,\mu(\cdot),\nu(\cdot)}$.
 Since conditions (\ref{26zv0})
 are equivalent to
 $$\sup\limits_{a>0} \frac1{a^{n(2+\frac{1}{ p}-\frac{1}{
q}-\frac \gamma n )}} \int\limits_{0}^{a} \nu_0(r)r^{n-1}dr\int\limits_{0}^{a}\mu_0(r)r^{n-1}
dr
< \infty$$
and
$$
\frac{\alpha}{ n}+\frac{1}{ p}> \frac{\gamma}{ n},\qquad \frac{\beta}{ n}+\frac{1}{ q'}> \frac{\gamma}{ n},\qquad\qquad \mbox{(cf.\, (\ref{26zv}))},
$$
(\ref{26zv0-0}) implies that  $A_{\gamma,\alpha, \beta} : L^{p,\tau_1}(\Bbb R^n)\to L^{q,\tau_2}(\Bbb R^n)$.


Let now the operator $A_{\gamma,\alpha, \beta}$ be bounded from $L^{p,\tau_1}(\Bbb R^n)$ to $L^{q,\tau_2}(\Bbb R^n)$. Then
(\ref{26zv0-0}) gives that $\gamma=\alpha+\beta+n\left(\frac1p-\frac1q\right)$.

First, we assume that
 $1<\tau_1\leq\tau_2\leq\infty$ and $\tau_1\ne\infty$.
Let us show that
$\alpha <n/p'$ and $\beta<n/q$.
Using  \cite[Th. 4.1]{NT} and taking into account that  $\left(L^{p,\tau_1}(\mathbb R^n)\right)'=L^{p',\tau'_1}(\mathbb R^n)$, we get
  \begin{eqnarray*}
1 &\gtrsim& \|A_{\gamma, \alpha, \beta} \|_{L^{p,\tau_1}(\mathbb R^n)\rightarrow L^{q,\tau_2}(\mathbb R^n)}
\\ &\gtrsim& \|
A_{\gamma, \alpha, \beta} \|_{L^{p,\tau_1}(\mathbb R^n)\rightarrow L^{q, \infty}(\mathbb R^n)}
\asymp \sup\limits_{ |E|>0}\frac{1}{|E|^\frac1{q'}}\left\|\int\limits_{E}
\frac1{|y|^\beta|x-y|^{n-\gamma}|x|^\alpha}\,
dy\right\|_{L^{p',\tau'_1}(\mathbb R^n)}.
\end{eqnarray*}
We further choose $E$ to be $B_1(\overline{0})$ or $B_1(\overline{3})$ and we get \begin{eqnarray*}
1 &\gtrsim&
\max\left\{
\frac{1}
{|B_1(\overline{0})|^\frac1{q'}}\Bigg\|
\chi_{\{x: |x|>2\}}(x)
\int\limits_{B_1(\overline{0})}
\frac{1}{|y|^\beta|x-y|^{n-\gamma}|x|^\alpha}\,
dy\Bigg\|_{L^{p',\tau'_1}(\mathbb R^n)}
\right.,
 \\&&\qquad\qquad\qquad
\left.\frac{1}{|B_1(3)|^\frac1{q'}}\Bigg\|\chi_{\{x: |x|<1\}}(x)
\int\limits_{B_1(\overline{3})}
\frac1{|y|^\beta|x-y|^{n-\gamma}|x|^\alpha}\,
dy\Bigg\|_{L^{p',\tau'_1}(\mathbb R^n)}
\right\}
\\&\gtrsim&
\max\Bigg\{\left\|
\chi_{\{x: |x|>2\}}(x)
|x|^{\gamma-n-\alpha}
\right\|_{L^{p',\tau'_1}(\Bbb R^n)}, \left\|
\chi_{\{x: |x|<1\}}(x)
|x|^{-\alpha}\right\|_{L^{p',\tau'_1}(\mathbb R^n)}
\Bigg\}
\\
&\gtrsim&
\max\left\{\left(\int_1^\infty t^{(1/p' -1+\gamma/n-\alpha/n)\tau'_1}\frac{dt}t\right)^{\frac1{\tau'_1}}, \left(\int_0^1 t^{(1/p'-\alpha/n)\tau'_1}\frac{dt}t\right)^{\frac1{\tau'_1}}\right\}.
 \end{eqnarray*}
Taking into account that $\tau'_1<\infty$ and $\gamma=\alpha+\beta+n\left(\frac1p-\frac1q\right)$, we have that convergence of integrals implies
 $\beta<n/q$ and $\alpha <n/p'$.

The case $\tau_1=\tau_2=\infty$  follows from the proof above and
$$
\|A_{\gamma, \alpha, \beta} \|_{L^{p,\infty}(\mathbb R^n)\rightarrow L^{q,\infty}(\mathbb R^n)}\geq\|A_{\gamma, \alpha, \beta} \|_{L^{p,2}(\mathbb R^n)\rightarrow L^{q,\infty}(\mathbb R^n)}.
$$

Let now $0<\tau_1\leq\tau_2<\infty$. Since the function $f(x)=(\tau_1/p)^{1/\tau_1}|W|^{-1/p}\chi_W(x)$ satisfies $\|f\|_{L^{p,\tau_1}(\Bbb R^n)}=1$,
 we get
\begin{eqnarray*}
1 &\gtrsim& \|A_{\gamma, \alpha, \beta} \|_{L^{p,\tau_1}(\mathbb R^n)\rightarrow L^{q,\tau_2}(\mathbb R^n)}
\\ &=&\sup_{\|f\|_{L^{p,\tau_1}(\Bbb R^n)}=1} \|
A_{\gamma, \alpha, \beta}f \|_{ L^{q, \tau_2}(\mathbb R^n)}
\geq  \sup\limits_{ |W|>0}\frac{1}{|W|^\frac1{p}}\left\|\int\limits_{W}
\frac1{|y|^\beta|x-y|^{n-\gamma}|x|^\alpha}\,
dx\right\|_{L^{q,\tau_2}(\mathbb R^n)}.
 \end{eqnarray*}
Repeating the arguments gives $\beta<n/q$ and  $\alpha <n/p'$.

(B). Let $0<\tau\le 1$ and (\ref{26zv1}) be true.
 Since
\begin{eqnarray*}
\|A_{\gamma, \alpha, \beta}\|_{L^{p, \tau}(\Bbb R^n)\to L^{q,\infty}(\Bbb R^n)}
&\asymp&
\|A_{\gamma, \alpha, \beta}\|_{L^{p, 1}(\Bbb R^n)\to L^{q,\infty}(\Bbb R^n)}
\\&\asymp& \sup_{{}_{}^{|E|>0, |W|>0}} \frac{1}{ {|E|^{1/q'}}}
\frac{1}{ {|W|^{1/p}}}
\int\limits_E
\int\limits_W
\frac1{|y|^\beta|x-y|^{n-\gamma}|x|^\alpha}\,
dx\, dy
,
 \end{eqnarray*}
 using Lemmas \ref{L1} and \ref{L3}, we get
$
\|A_{\gamma, \alpha, \beta}\|_{L^{p, 1}(\Bbb R^n)\to L^{q,\infty}(\Bbb R^n)}
\lesssim \mathcal{D}$, where $\mathcal{D}$ is defined in Lemma \ref{L3} and $\mu(x)=|x|^{-\alpha}$, $\nu(x)=|x|^{-\beta}$, and $K(x)=|x|^{\gamma-n}$.
Note that conditions (\ref{3zv})-(\ref{4zv}) hold since $\alpha /n\leq 1/p'$ and  $\beta/n\leq 1/q.$
Finally, $\|A_{\gamma, \alpha, \beta}\|_{L^{p, 1}(\Bbb R^n)\to L^{q,\infty}(\Bbb R^n)}\lesssim\mathcal{D}<\infty$ because of $\gamma=\alpha+\beta+n\left(\frac1p-\frac1q\right)$.

To prove the inverse part, we have
\begin{eqnarray*}
1 &\gtrsim& \|A_{\gamma, \alpha, \beta} \|_{L^{p,1}(\mathbb R^n)\rightarrow L^{q,\infty}(\mathbb R^n)}
\\
&\gtrsim& \max\Bigg\{\sup_{0<\delta\leq1}\frac{1}{|B_\delta(\overline{0})|^\frac1{q'}|B_1(\overline{3})|^{\frac1p}}\int\limits_{B_1(\overline{3})}
\int\limits_{B_\delta(\overline{0})}
\frac{dxdy}{|y|^\beta|x-y|^{n-\gamma}|x|^\alpha},
\\
&&\qquad\quad
\sup_{0<\delta\leq1}\frac{1}{|B_1(\overline{3})|^\frac1{q'}|B_\delta(\overline{0})|^{\frac1p}}\int\limits_{B_\delta(\overline{0})}
\int\limits_{B_1(\overline{3})}
\frac{dxdy}{|y|^\beta|x-y|^{n-\gamma}|x|^\alpha}\,
 \Bigg\}
\\
&\gtrsim& \max\Bigg\{\sup_{0<\delta\leq1}\frac{1}{|B_\delta(\overline{0})|^\frac1{q'}}\int\limits_{B_\delta(\overline{0})}
\frac1{|y|^\beta}\,
dy, \;\;\sup_{0<\delta\leq1}\frac{1}{|B_\delta(\overline{0})|^{\frac1p}}\int\limits_{B_\delta(\overline{0})}
\frac1{|x|^\alpha}\,
dx \Bigg\}.
\end{eqnarray*}
This gives $\alpha /n\leq 1/p'$ and $\beta/n\leq 1/q.$
Necessity of the condition  $\gamma=\alpha+\beta+n\left(\frac1p-\frac1q\right)$ follows from
$$
1 \gtrsim \|A_{\gamma, \alpha, \beta} \|_{L^{p,1}(\mathbb R^n)\rightarrow L^{q,\infty}(\mathbb R^n)}
\gtrsim \sup_{0<\delta<\infty}\frac{1}{|B_\delta(\overline{0})|^\frac1{q'}|B_\delta(\overline{0})|^{\frac1p}}\int\limits_{B_\delta(\overline{0})}\int\limits_{B_\delta(\overline{0})}
\frac{dxdy}{|y|^\beta|x-y|^{n-\gamma}|x|^\alpha}
$$
and the homogeneity argument.
\end{proof}

\vskip 0.8 cm

\section{Final remarks}\label{final}

\vskip 0.2 cm

{\bf 1.} Theorem A provides sufficient conditions for weighted Young's inequality
$L^{p}_\alpha(\Bbb R^n)*L^{\theta}_{s}(\Bbb R^n)\subset L^{q}_{-\beta}(\Bbb R^n)$ to hold.
Certain necessary conditions, which differ from sufficient ones, were obtained by Bui \cite[Th. 2.1]{bui}.
He posed the problem (\cite[p.32]{bui}) of finding necessary and sufficient conditions for weighted Young's inequality.
We give a solution of this problem.

\begin{thma'}
Suppose  $1 < p,q < \infty$ and $1 < \theta \le \infty$.
Then
 \begin{eqnarray}\label{emb11}
 L^{p}_\alpha(\Bbb R^n)*L^{\theta}_{s}(\Bbb R^n)\subset L^{q}_{-\beta}(\Bbb R^n),\qquad 
  \frac{1}{q}\le \frac{1}{p}+\frac{1}{\theta},
\end{eqnarray}
if and only if
\begin{eqnarray}\label{parameters11}
 \frac{1}{q}=\frac{1}{p}+\frac{1}{\theta}+\frac{\alpha+\beta+s}{n}-1,
\end{eqnarray}
\begin{eqnarray}\label{parameters21}
\qquad
\alpha<n/p',\quad  \beta<n/q,\quad s<n/\theta'\qquad\qquad \mbox{if}\quad \theta<\infty
\end{eqnarray}
$$
\alpha<n/p',\quad  \beta<n/q,\quad 0<s<n\qquad \,\quad\mbox{if}\quad \theta=\infty
 \leqno{(\ref{parameters21})'}
$$
and
\begin{eqnarray}\label{parameters31}
\alpha+  \beta\ge 0,\quad \alpha+  s\ge 0, \quad  \beta+s\ge 0.
\end{eqnarray}
\end{thma'}

\begin{remark1}{{
Necessity of condition (\ref{parameters31}) was recently proved in \cite[Rem. 4.4]{mv} and \cite[Rem. 2.1]{denapoli}.
Our proof uses different argument.}
Note also that in case of $\theta=\infty$ condition (\ref{parameters31}) is equivalent to the following condition:
$$
\alpha+  \beta\ge 0.
 \leqno{(\ref{parameters31})'}
$$
This follows from  (\ref{parameters11}) and (\ref{parameters21})$'$.
}
\end{remark1}

\begin{proof}
Sufficiency for the case $\theta<\infty$ was proved by Kerman (Theorem A). For the case $\theta=\infty$ see \cite[Th. 3.4]{Kerman}, \cite[Th. 2.2 (ii)]{bui}, and \cite[Sect. 3]{be}.

Let us prove the necessity part.
First, by scaling argument, we have
necessity of condition (\ref{parameters11}) (see \cite[Th. 2.1]{bui}).

Second, let $1<\theta\le \infty$ and (\ref{emb11}) hold true.
 Take $\widetilde{K}(x)=K(x)|x|^s\in L^\theta(\Bbb R^n)$, i.e., $K\in L^\theta_s(\Bbb R^n)$. Then  the operator
 $$
  A_{\tilde{K}, \alpha, \beta,s} f(y)=\frac{1}{|y|^{\beta}}\int\limits_{\mathbb R^n}\frac{\widetilde{K}(x-y)f(x)}{|x-y|^s|x|^\alpha}\,
dx
$$
is bounded from $L^p(\Bbb R^n)$ to $L^q(\Bbb R^n)$. Using \cite[Th. 4.1]{NT}, we get
\begin{eqnarray*}
 \|A_{\tilde{K}, \alpha, \beta,s} \|_{L^{p}(\mathbb R^n)\rightarrow L^q(\mathbb R^n)}&\geq& \|
A_{\tilde{K}, \alpha, \beta,s} \|_{L^{p}(\mathbb R^n)\rightarrow L^{q, \infty}(\mathbb R^n)}
\nonumber
\\
&\asymp& \sup\limits_{ |E|>0}\frac{1}{|E|^\frac1{q'}}
\Big\|\int\limits_{E}
\frac{\widetilde{K}(x-y)}{|y|^\beta|x-y|^s|x|^\alpha}\,
dy\Big\|_{(L^{p}(\mathbb R^n))'}
\\
&\gtrsim&
 \sup\limits_{ |E|>0}\frac{1}{|E|^\frac1{q'}}\left(\;\int\limits_{\Bbb R^n}
\Big|\int\limits_{E}
\frac{\widetilde{K}(x-y)}{|y|^\beta|x-y|^s|x|^\alpha}\,
dy\Big|^{p'}
dx\right)^{1/p'}.
\end{eqnarray*}
Let an integer $N$ be sufficiently large. Put
 $$\widetilde{K}(x):=\chi_{B_2({\overline{N}})}(x)\in L^{\theta}(\Bbb R^n).$$
Then
\begin{eqnarray*}
1&\gtrsim&
 \|A_{\tilde{K}, \alpha, \beta,s} \|_{L^{p}(\mathbb R^n)\rightarrow L^q(\mathbb R^n)}
\gtrsim
\frac{1}{|B_{1}(-{\overline{N}})|^\frac1{q'}}\left(\int\limits_{B_1(\overline{0})}
\Big|\int\limits_{B_1(-{\overline{N}})}
\frac{\tilde{K}(x-y)}{|y|^\beta|x-y|^s|x|^\alpha}\,
dy\Big|^{p'}
dx\right)^{1/p'}
\\
&\gtrsim&
\left(\int\limits_{B_1(\overline{0})}|x|^{-\alpha p'}dx\right)^{1/p'}\int\limits_{B_1({\overline{N}})}\frac{dy}{|y|^{s+\beta}}.
\end{eqnarray*}
In the last expression, the first integral converges only if
$\alpha< n/p'$ while 
 the second integral is equivalent to $N^{-(\beta+s)}$ which is bounded by a constant only if $\beta+s\ge 0$.


Taking into account that
$$
\|A_{\tilde{K}, \alpha, \beta,s} \|_{L^{p }(\mathbb R^n)\rightarrow L^{q}(\mathbb R^n)}=\|A_{\tilde{K}, \beta, \alpha,s} \|_{L^{q'}(\mathbb R^n)\rightarrow L^{p'}(\mathbb R^n)},
 $$
 we also get that $\alpha+s\ge 0$ and $\beta<n/q$.

To prove necessity of
conditions $\alpha+\beta\geq 0$ and $s< n/\theta'$, we consider the operator
$$
 \left( A_{\tilde{f}, \alpha, \beta,s} K\right)(y)=\frac{1}{|y|^{\beta}}\int\limits_{\mathbb R^n}\frac{\tilde{f}(x-y)K(x)}{|x-y|^\alpha|x|^s}\,
dx
$$
from $L^\theta(\Bbb R^n)$  to $L^q(\Bbb R^n)$  with the kernel
$$
\frac{\tilde{f}(x-y)}{|y|^\beta|x-y|^\alpha|x|^s},
$$
where $\tilde{f}(x)=f(x)|x|^\alpha\in L^p(\Bbb R^n)$, and proceed similarly.

Finally, we have to show that if $\theta =\infty$, then $s>0$.
  Let $L^{p}_\alpha(\Bbb R^n)*L^{\infty}_{s}(\Bbb R^n)\subset L^{q}_{-\beta}(\Bbb R^n),$ $1<p\le q<\infty$.
We define
$$
K(x):=|x|^{-s}\in L^{\infty}_s(\Bbb R^n),  \;\;\;\;\;   f(x):=\chi_{\Bbb R^n\diagdown B_2(\overline{0})}(x)|x|^{-\alpha-n/p}(\ln |x|)^{-2/p}\in L^p_\alpha(\Bbb R^n).
$$
Taking into account that  $\alpha<n/p'$, we notice that $s$ has to be positive
 in order that the integral
$$
(K*f)(y)=\int\limits_{\Bbb R^n\diagdown B_2(\overline{0})}\frac{1 }{|y-x|^{s}|x|^{\alpha+n/p}(\ln |x|)^{2/p}}dx
$$
converges.
\end{proof}

\vskip 0.8 cm

{\bf 2.} The Plancherel-Polya-Nikol'skii inequality plays an important role in the theory of function spaces (see, e.g., \cite{mv, tribel}) and approximation theory (see, e.g., \cite{devore, ditzian}).
The following result  provides sufficient conditions for a two-weight version of this inequality, partially answering the question posed in \cite[Rem. 4.2]{mv}. Let $\mathcal{F}$ and  $\mathcal{F}^{-1}$ denote the Fourier transform and its inverse.
\begin{corollary}\label{nikolll}
Let $p,q,r, \mu$,  and $\nu$ satisfy conditions of Theorem \ref{TTT}. Let
$Q_d$ is a rectangular parallelepiped in $\Bbb R^n$
with edges of lengths $d=(d_1,...,d_n)$. If \, $supp \,({\mathcal{F}{f}}) \subseteq Q_d$, then
$$
\|f\|_{L^q(\nu;\Bbb R^n)}\lesssim \mathcal{G}\,\left(\prod_{j=1}^nd_j \right)^{1/{r'}}\|f\|_{L^p(\mu^{-1};\Bbb R^n)},
$$
where $\mathcal{G}$ is given by (\ref{21}).
\end{corollary}
\begin{remark1}
For the power weights necessary and sufficient conditions for the Plancherel-Polya-Nikol'skii inequality were proved in \cite{mv}, see
Proposition 4.1 and Remark 4.2.
Plancherel-Polya-Nikol'skii's inequality with $A_\infty$ weights were studies in \cite{bui1, mast}.
The weighted version of Nikol'skii's inequality  for trigonometric polynomials was also studied in
\cite[Cor. 5.2]{NT}.
\end{remark1}

The proof of Corollary \ref{nikolll} immediately follows from the following result. 


\begin{corollary}\label{cor4}
Let $Q_d$ be a rectangular parallelepiped in $\Bbb R^n$ with edges of lengths $d=(d_1,...,d_n)$ and
$$
S_{Q_d}(f):=\mathcal{F}^{-1}\chi_{Q_d} \mathcal{F}f.
$$
{\textnormal{(A)}}. Let $p,q,r$ and $\mu, \nu$ satisfy conditions of Theorem \ref{TTT}.
If
$\mathcal{G} < \infty,$  where $\mathcal{G}$ is defined by (\ref{21}),
then
\begin{eqnarray}\label{C4-1}
\|S_{Q_d}\|_{L^p(\mu^{-1};\Bbb R^n)\to L^q(\nu;\Bbb R^n)}\leq C \mathcal{G}\left(\prod_{j=1}^nd_j \right)^{1/{r'}},
\end{eqnarray}
where $C$
is independent of  $d$.
\\{\textnormal{(B)}}. Let $1<p,q,r<\infty$.
If
$$
\|S_{Q_d}\|_{L^p(\mu^{-1};\Bbb R^n)\to L^q(\nu;\Bbb R^n)}<\infty,
$$
then
\begin{eqnarray}\label{C4-3}
\mathcal{K} :=\sup_{{\frac{\pi}{4d}\leq t\leq\frac\pi{2d}}}\sup\limits_{Q_t}\frac{\nu(Q_t) \mu(Q_t)}{|Q_t|^{1+\frac1r+\frac1p - \frac1q }}< \infty
\end{eqnarray}
and
\begin{eqnarray}\label{C4-2}
\|S_{Q_d}\|_{L^p(\mu^{-1};\Bbb R^n)\to L^q(\nu;\Bbb R^n)}\geq C \mathcal{K}\left(\prod_{j=1}^nd_j \right)^{1/{r'}},
\end{eqnarray}
where $C$ is independent of  $d$.
\end{corollary}

\begin{proof}
Taking into account that
$
S_{Q_d}(f)=
\left(\mathcal{F}^{-1}\chi_{Q_d}\right)*f,
$
and
$$
\mathcal{F}^{-1}\chi_{Q_d}(x)=\prod_{j=1}^n\frac{\sin (d_jx_j/2)}{x_j/2},
$$
we get from Theorem \ref{TTT} that
\begin{eqnarray*}
\|S_{Q_d}(f)\|_{L^q(\nu;\Bbb R^n)} &\lesssim& \mathcal{G}\left\|\prod_{j=1}^n\frac{\sin (d_jx_j/2)}{x_j/2}\right\|_{L^{r, \infty}(\mathbb R^n)}\|f\|_{L^p(\mu^{-1};\Bbb R^n)}\\&\lesssim& \mathcal{G}\left(\prod_{j=1}^nd_j \right)^{\frac1{r'}}\|f\|_{L^p(\mu^{-1};\Bbb R^n)},
\end{eqnarray*}
i.e., (\ref{C4-1}) follows.

Let now
$
\|S_{Q_d}\|_{L^p(\mu^{-1};\Bbb R^n)\to L^q(\nu;\Bbb R^n)}<\infty.
$
Then, defining  the operator
$$S_{Q_d, \mu,\nu}(f,y)=
\nu(y)\int\limits_{\Bbb R^n}\prod_{j=1}^n\frac{\sin (d_j(y_j-x_j)/2)}{(y_j-x_j)/2}\mu(x)f(x)dx,
$$
we get
\begin{align*}
&
\|S_{Q_d}\|_{L^p(\mu^{-1};\Bbb R^n)\to L^q(\nu;\Bbb R^n)}=\|S_{Q_d, \mu,\nu}\|_{L^p(\Bbb R^n)\to L^q(\Bbb R^n)}\gtrsim \|S_{Q_d, \mu,\nu}\|_{L^{p,1}(\Bbb R^n)\to L^{q,\infty}(\Bbb R^n)}
\\
&\asymp \sup_{{}_{}^{|E|>0, |W|>0}} \frac{1}{ {|E|^{1/q'}}}
\frac{1}{ {|W|^{1/p}}} \left|
\int\limits_E \nu(y) \int\limits_W
\mu(x) \prod_{j=1}^n\frac{\sin (d_j(y_j-x_j)/2)}{(y_j-x_j)/2} dxdy\right|.
\end{align*}
Since  for $x,y\in Q_t$ we have that $y-x\in Q_{2t}(\overline{0})$, where $Q_{2t}(\overline{0})$ is the rectangular parallelepiped with edges of lengths $2t=(2t_1,...,2t_n)$
centered   at $\overline{0}=(0, \ldots, 0)$, then
\begin{align*}
\|S_{Q_d}\|_{L^p(\mu^{-1};\Bbb R^n)\to L^q(\nu;\Bbb R^n)}&\gtrsim
\sup_{\frac{\pi}{4d}\leq t\leq\frac\pi{2d}}\sup\limits_{Q_t}\frac1{|Q_t|^{\frac1p + \frac1{q'} }}
 \int\limits_{Q_t} \nu(y) \int\limits_{Q_t}
\mu(x) \prod_{j=1}^n\frac{\sin (d_j(y_j-x_j)/2)}{(y_j-x_j)/2} dxdy
\\
&\gtrsim
 \;
\prod_{j=1}^n{d_j}\sup_{\frac{\pi}{4d}\leq t\leq\frac\pi{2d}}\sup\limits_{Q_t}\frac1{|Q_t|^{\frac1p + \frac1{q'} }}
 \int\limits_{Q_t} \nu(y) \int\limits_{Q_t}
\mu(x)  dxdy
\\
&\gtrsim \;
\left(\prod_{j=1}^n d_j\right)^{1/r'}\sup_{\frac{\pi}{4d}\leq t\leq\frac\pi{2d}}\sup\limits_{Q_t}\frac{\nu(Q_t) \mu(Q_t)}{|Q_t|^{1+\frac1r+\frac1p - \frac1q }}.
\end{align*}
\end{proof}

For power weights $\mu(x)=|x|^{-\alpha}$ and $\nu(x)=|x|^{-\beta}$, we obtain necessary and sufficient conditions for boundedness of the operator $S_{Q_d}$.
\begin{corollary}\label{cor6}
Let  $1 < p , q,s < \infty$ and $\alpha\geq 0, \beta\geq 0$.
Then
\begin{eqnarray}\label{C6-1}
\|S_{Q_d}(f)\|_{L^q_{-\beta}(\Bbb R^n) }\leq C \,
\left(\prod_{j=1}^nd_j \right)^{1/{s}}
\|f\|_{L^p_\alpha(\Bbb R^n)} \qquad \forall\quad d_1,\ldots,d_n>0,
\end{eqnarray}
where $Q_d$ is a rectangular parallelepiped in $\Bbb R^n$ with edges of lengths $d=(d_1,\ldots,d_n)$
and $C$ is independent of  $d_1,\ldots,d_n$,
if and only if
\begin{eqnarray}\label{C6-2}
\frac1s=\frac{1}{p}-\frac{1}{q}+\frac{\alpha+\beta}{n},
 \quad \alpha<n/p',\quad   \beta<n/q.
\end{eqnarray}
\end{corollary}

\begin{proof}
Let first (\ref{C6-2}) hold.
It is easy to see that
all conditions of Theorem \ref{TTT} are fulfilled  and $\mathcal{G} \lesssim 1$. 
Therefore, Corollary  \ref{cor4} (A) gives  (\ref{C6-1}).

On the other hand,  (\ref{C6-1}) implies
$$
\sup_{d_1,\ldots, d_n>0}
\left(\prod_{j=1}^nd_j \right)^{-1/{s}}
\|S_{Q_d}\|_{L^p_{\alpha}(\Bbb R^n)\to L^q_{-\beta}(\Bbb R^n)}\lesssim 1.
$$
Then using (\ref{C4-2}) with $r=s'$ and $d_1=\ldots=d_n=z$
we arrive at
$
 z^{n(\frac 1p-\frac 1q-\frac 1s)+\alpha+\beta} \lesssim \mathcal{K} \lesssim 1
 $ for any ${z>0}$.
This yields
$\frac1s=\frac{1}{p}-\frac{1}{q}+\frac{\alpha+\beta}{n}$.

Using a method similar to the proof of Corollary \ref{C111}, one can verify that conditions  $\alpha<n/p'$ and $\beta<n/q$ hold.
 Let us show, for example, that $\beta<n/q$.
Indeed, let $C_{\xi}(\overline{z})$ be a cube with
the edge length
 $\xi$ 
centered  at $\overline{z}=(z_1, \ldots, z_n)$.
Fix $d_1, \ldots, d_n>0$
and
set $d_0:=\max_{1\leq i\leq n} d_i$. Notice that
if
 $x\in C_\delta(\overline{0})$, where  $\delta<\pi/4d_0$, and $y\in C_{\pi/{(4d_0)}}(\overline{\pi/d})$, then $\frac{\pi}{4}\leq\frac{d_i|x_i-y_i|}2\leq \frac{3\pi}{4}$, $i=1,\ldots,n$. Therefore,

\begin{align*}
&\|S_{Q_d}\|_{L^p(\mu^{-1};\Bbb R^n)\to L^q(\nu;\Bbb R^n)}
\gtrsim
\Big|
C_{\pi/({4d_0})}(\overline{\pi/d})
\Big|^{-\frac1p}
\Big\|
|y|^{-\beta} \int\limits_{
C_{\pi/({4d_0})}(\overline{\pi/d})
}
|x|^{-\alpha} \prod_{j=1}^n\frac{\sin (d_j(y_j-x_j)/2)}{(y_j-x_j)/2} dx\Big\|_{L_{q}(\Bbb R^n)}
\\
&\qquad \ge
\Big|
C_{\pi/({4d_0})}(\overline{\pi/d})
\Big|^{-\frac1p}
\left(\int\limits_{
C_\delta(\overline{0})
}\Bigg(|y|^{-\beta} \int\limits_{C_{\pi/({4d_0})}(\overline{\pi/d})}
|x|^{-\alpha} \prod_{j=1}^n\frac{\sin (d_j(y_j-x_j)/2)}{(y_j-x_j)/2} dx\Bigg)^qdy\right)^{1/q}
\\& \ \ \ \ \ \ \gtrsim
\left(\int\limits_{C_\delta(\overline{0})}|y|^{-\beta q} dy\right)^{1/q}.
\end{align*}
This implies $\beta<n/q$.
\end{proof}

\vskip 0.5 cm

\end{document}